\documentclass[hidelinks, a4,11pt]{article}

\usepackage{amsmath}
\usepackage{amssymb}
\usepackage{amsthm}
\usepackage{lineno}
\usepackage{graphicx}
\usepackage{caption}
\usepackage{subcaption}
\usepackage{blkarray}
\usepackage{appendix}
\usepackage{enumerate}
\usepackage[font=small,skip=0pt]{caption}
\usepackage{xcolor}
\usepackage{verbatim}
\usepackage{geometry}
\usepackage{enumerate}
\usepackage{wrapfig}
\usepackage{color,soul}
\usepackage{float}
\usepackage{lineno}
\usepackage[stable]{footmisc}
\usepackage{lineno}
\usepackage{blindtext}
\usepackage{calc}
\usepackage{upgreek}
\usepackage{amsmath}
\usepackage{amssymb}
\usepackage{amsthm}
\usepackage{lineno}
\usepackage{graphicx}
\usepackage{caption}
\usepackage{subcaption}
\usepackage[font=small,skip=0pt]{caption}
\usepackage{xcolor}
\usepackage{verbatim}
\usepackage{geometry}
\usepackage[colorinlistoftodos]{todonotes}
\usepackage{wrapfig}
\usepackage{color,soul}
\usepackage{float}
\usepackage{url}
\usepackage[colorinlistoftodos]{todonotes}
\usepackage{wrapfig}
\usepackage{color,soul}
\usepackage{float}
\usepackage{multirow}
\usepackage[utf8]{inputenc}

\usepackage{hyperref}
\usepackage{cancel}

\DeclareMathAlphabet\mathbfcal{OMS}{cmsy}{b}{n}
\DeclareMathAlphabet\mathbfscr{OMS}{cmsy}{b}{n}

\newtheorem{thm}{Theorem}[section]
\newtheorem{rem}{Remark}[section]
\newtheorem{mydef}{Definition}[section]
\newtheorem{lem}{Lemma}[section]
\newtheorem{prop}{Proposition}[section]

\newtheorem{ex}{Example}[section]

\numberwithin{equation}{section}

\def \E{{\rm I\kern-0.16em E}}
\def\P{{\rm I\kern-0.16em P}}
\def\F{{\rm I\kern-0.16em F}}
\def\B{{\rm I\kern-0.16em B}}
\def\C{{\rm I\kern-0.46em C}}
\def\D{{\rm I\kern-0.50em D}}
\newcommand{\Z}{\mathbb{Z}}

\newcommand{\R}{\mathbb{R}}
\newcommand{\N}{\mathbb{N}}

\newcommand{\EXP}{\textbf{Exp}}

\newcommand{\GID}{\text{GID}}

\newcommand{\sgn}{\text{sgn}}

\newcommand{\dd}{\textnormal{d}}
\newcommand{\ee}{\textnormal{e}}

\usepackage{ae}

\usepackage{color}
\newcommand{\gr}[1]{{\color{blue} #1}}
\newcommand{\gdr}[1]{{\color{red} #1}}

\makeatletter
\newcommand*\bigcdot{\mathpalette\bigcdot@{.5}}
\newcommand*\bigcdot@[2]{\mathbin{\vcenter{\hbox{\scalebox{#2}{$\m@th#1\bullet$}}}}}
\makeatother

\makeatletter
\newcommand{\ostar}{\mathbin{\mathpalette\make@circled\star}}
\newcommand{\make@circled}[2]{%
	\ooalign{$\m@th#1\smallbigcirc{#1}$\cr\hidewidth$\m@th#1#2$\hidewidth\cr}%
}
\newcommand{\smallbigcirc}[1]{%
	\vcenter{\hbox{\scalebox{0.77778}{$\m@th#1\bigcirc$}}}%
}
\makeatother

\usepackage{todonotes}

\title{ {\bf Probabilistic Cauchy Functional Equations}}

\author{Ehsan Azmoodeh\footnote{University of Liverpool, Department of Mathematical Sciences, L69 7ZL Liverpool, United Kingdom; E-mail address: ehsan.azmoodeh@liverpool.ac.uk}\,, \,
Noah Beelders\footnote{E-mail address: Noah.Beelders@liverpool.ac.uk}\, and Yuliya Mishura\footnote{Department of Probability, Statistics and Actuarial Mathematics, Taras Shevchenko National University of Kyiv, M\"alardalen University, E-mail address: yuliyamishura@knu.ua}
}

\date{ }

\begin{document}

\maketitle	
%\linenumbers	
	\begin{abstract}
	In this short note, we introduce probabilistic Cauchy functional equations, specifically, functional equations of the following form:
	
	\begin{equation*}
	 f(X_1 + X_2)     \stackrel{\dd}{=}	f(X_1) + f(X_2),
	\end{equation*}
	where  $X_1$ and $X_2$ represent two independent identically distributed real-valued random variables governed by a distribution $\mu$ having appropriate support on the real line. The symbol $\stackrel{d}{=}$ denotes equality in distribution. When $\mu$ follows an exponential distribution, we provide sufficient (regularity) conditions on the function $f$ to ensure that the unique measurable solution to the above equation is solely linear. Furthermore, we present some partial results in the general case, establishing a connection to integrated Cauchy functional equations.
	
	\end{abstract}
	

	\section{Introduction and motivation }\label{Sec:Main-Theorems}
	
	Functional equations serve as a cornerstone in various mathematical disciplines, by playing a pivotal role in both pure and applied mathematics. Among these, the (additive) Cauchy functional equation
	\begin{equation}\label{eq:CauchyEq}
	f(x+y) = f(x)+f(y), \quad x,y \in \R,
	\end{equation}	
	 stands out as a fundamental and intriguing problem that has garnered significant attention since its inception in the 19th century. In fact, Cauchy in 1818 proved that any \textbf{continuous} solution of  equation \eqref{eq:CauchyEq} is linear; that is, $f(x) =cx$ for some constant $c$. The argument is straightforward: first verify linearity for the case of rationals and then pass to the limit using continuity. In fact, the assumption of continuity can be much further relaxed. Lebesgue \cite{Lebesque} showed that any measurable solution to equation \eqref{eq:CauchyEq} is linear. There is a myriad of interesting literature examining varied conditions on $f$ such that the only solution to equation \eqref{eq:CauchyEq} is of a linear form, e.g., \cite{Banach, Darboux1, Darboux2, Frechet, Sierpinski}.  The curious reader can find a comprehensive treatment of the regularity properties concerning functional equations in J\'arai's book \cite{Jarai}. Another fundamental aspect of the Cauchy functional equation \eqref{eq:CauchyEq} was discovered by Hamel in 1905 \cite{Hamel}. He showed that without any conditions on $f$, the axiom of choice could be used to demonstrate that there exist also nonlinear solutions of \eqref{eq:CauchyEq} for which he also found all such solutions.  The functional equation (1) has been generalized or modified in many other directions. A typical scenario is to change the domain of validity of the equation; for instance, to assume that $f$ satisfies equation \eqref{eq:CauchyEq} only for pairs $(x, y)$ belonging to a subset of $\R^2$. Additional information regarding equation \eqref{eq:CauchyEq}, along with numerous supplementary references, can be explored in works such as those by Acz\'el \cite{Aczel1} and Acz\'el-Dhombres \cite{Aczel2}. Moreover, the applications of equation \eqref{eq:CauchyEq} in diverse scientific domains are also extensively covered.\\
	
	  In this paper, we are interested in a probabilistic version of equation \eqref{eq:CauchyEq}. For the ease of presentation, we only consider the one dimensional case.  Let $\mu$ be an absolutely continuous probability measure (w.r.t.~the Lebesgue measure) on the real line. The instances of discrete distributions are excluded since the Cauchy additive functional equation on discrete spaces such as $\Z_+$ is solved trivially. Throughout our presentation, we shall denote by $S(\mu)$ the support of $\mu$ and additionally assume that $S(\mu) = \R, \, \R_+$ or $\R_-$ otherwise it is clearly stated. Also,  the probability space $(\Omega, \mathcal{F},\P)$ is fixed and all the random objects are assumed to be defined on it.
	 \begin{mydef}[Probabilistic Cauchy functional equation]
	 Fix an integer $n\ge 2$. Let $\mu$ be an absolutely continuous probability measure on the real line. We say that a function $f: S(\mu) \to \R$ satisfies the $n$-summands probabilistic Cauchy functional equation with respect to $\mu$ whenever
	 \begin{equation}\label{eq:PCE}
	 	f(X_1+\dots+X_n)  \stackrel{\dd}{=}      f(X_1) + \dots+ f(X_n),
	 	\end{equation}
 	where  $X_1, \dots, X_n$ are $n$ independent identically distributed real-valued random variables whose laws equal to $\mu$. We simply say that $f$ satisfies  the probabilistic Cauchy equation with respect to $\mu$ whenever equation \eqref{eq:PCE} holds with $n=2$.
	 \end{mydef}

	In contrast to the deterministic Cauchy functional equation, adding more summands in the probabilistic Cauchy functional equation changes its distributional properties, and hence we let
	
	 \begin{equation}
	 \begin{aligned}
	 	\mathbfcal{PCFE} (\mu;n) & := \Big \{  f:S(\mu)\to \R \text{ satisfying equation }  \eqref{eq:PCE}         \Big \}, \text{ and } \\
	 	\mathbfcal{PCFE} (\mu;2) &= \mathbfcal{PCFE} (\mu).
	 	\end{aligned}	
 	\end{equation}
	
 Here, we have to remove a few ambiguities. First, any solution $f$ to equation \eqref{eq:PCE} remains a valid solution under any suitable modification on every null set (w.r.t Lebesgue measure).  Hence, any solution $f$ to equation \eqref{eq:PCE} induces an equivalent solution class  $[f]$ (under obvious equivalence relation $f \sim g$ iff $f(x) = g(x)$ a.e. Lebesgue measure).Hence,
    the continuity assumption would be a natural condition that could be imposed on the solution $f$ to combine all equivalent solutions into a single solution. Second, when the support $S(\mu)$ is not the entire real line, any solution $f \in \mathbfcal{PCFE}(\mu;n)$ can have its domain extended to the whole real line in many different ways with no harm. In such cases, we are solely interested in the structure of the solutions of the probabilistic Cauchy equation on the support of the underlying probability measure, and hence specify the domain of solution $f \in \mathbfcal{PCFE}(\mu;n)$ as $S(\mu)$. A typical instance is when $\mu \sim \EXP (\lambda=1)$ is an exponential distribution whose support is $\R_+$ (we will examine it in more details later on).   Now, we ask ourselves the following natural question sharing the same spirit of the classical Cauchy function equation \eqref{eq:CauchyEq}:

 \begin{center}
 \textit{What regularity conditions dictate that any solution $f \in \mathbfcal{PCFE} (\mu;n)$ takes the linear form $f(x) = cx$ (for some constant $c$) almost everywhere or everywhere?}
 \end{center}

	In this note, we consider the above question with a particular emphasis on when $\mu$ is an exponential distribution. We address in Section \ref{sec:ConnectionPhaseType}, one of the main reasons for choosing the exponential distribution is that the above central question arose naturally during the study of a useful property in the class of (fractional) phase-type distributions. In Section \ref{sec:ExpCase}
     we present several sufficient regularity conditions on $f$ that guarantee the linear form of the unique solution to the probabilistic Cauchy equation  \eqref{eq:PCE}.   Section \ref{sec:GeneralCase} compiles a few partial results for a general probability measure $\mu$ with connections to integrated Cauchy functional equations.

	\subsection{Almost sure Cauchy functional equation}
	There exist other stochastic versions of the Cauchy functional equation \eqref{eq:CauchyEq} proposed in the existing literature. For example, in \cite{Mania}, the authors  studied solutions of the stochastic Cauchy functional equation
	
	\begin{equation}\label{eq:PCFE-as}
	 f(X_1 + X_2) = f(X_1)+f(X_2), \quad a.s.,
	 \end{equation}	
	where $X_1, X_2 \sim \mathcal{N}(0,1)$ are two independent standard Gaussian random variables. Note that, clearly, equation \eqref{eq:PCFE-as} is much stronger than our probabilistic Cauchy functional equation \eqref{eq:PCE} when $\mu = \mathcal{N}(0,1)$. However, their primary goal differs from ours in that they connect the almost sure Cauchy functional equation \eqref{eq:PCFE-as} to a martingale problem. In fact, they proved the following interesting observation. Let $W = (W_t)$ be a Brownian motion w.r.t.~its natural filtration. Then, the stochastic process $M_t= f(W_t)$ is a martingale where $f$ is a solution to equation \eqref{eq:PCFE-as} iff $f(x) =cx$ almost everywhere w.r.t.~Lebesgue measure. Additionally, if one imposes the extra regularity assumption of "right-continuity" on the martingale $M$, then the "almost everywhere" condition in the former statement can be removed. Their study was inspired by the work of Smirnov \cite{Smirnov} who showed that a weak version of Bernstein’s characterization of the Gaussian distribution implies the local integrability of a measurable solution of the Cauchy functional equation \eqref{eq:CauchyEq}.

	\subsection{Connection with  the inhomogeneous phase-type distributions}\label{sec:ConnectionPhaseType}
	Let's consider the parameterised exponential family of distributions; that is
	\begin{equation}\label{eq:EXP-family}
		\mathbfcal{E}:= \{ F(\theta; x): = 1 - e^{-\theta x} \,: \, \theta >0, x \in (0,\infty)  \}.
	\end{equation}
	It is a classical result (or can be seen by direct computation) that the exponential family $\mathbfcal{E}$ enjoys the following interesting property: for $n \in \N$ and every $\theta \in (0,\infty)$,
	\begin{equation}\label{eq:EXP-property}
		\underbrace{F(\theta;\cdot) \star \cdots \star  F(\theta;\cdot)}_{n \text{ times }}= F^{\star n} ( \theta;\cdot)   =  \frac{(-\theta)^n }{(n-1)!}  \,    \frac{\partial^{n-1}}{\partial \theta^{n-1}} G (\theta;\cdot),
	\end{equation}
	where $G(\theta;\cdot):=-\theta^{-1} F(\theta;\cdot)$ and $\star$ stands for the convolution operator. The beauty of relation \eqref{eq:EXP-property} is that it reduces the derivation of the distribution of the sum of $n$ independent exponential random variables to simple differentiation w.r.t.~the parameter variable. Now, let $X(\theta) \sim F(\theta;\cdot)$ be a random variable corresponding to the distribution function $F(\theta;\cdot)$ for each $\theta \in \Theta$. In addition, let $h:\R \rightarrow \R$ be a bijective function, and consider the new random variable
	\begin{equation}\label{eq:h-family}
		Y(\theta) := h(X(\theta)) \sim F_h(\theta; \cdot) := F\big(\theta; h^{-1}(\cdot)\big), \quad \forall \, \theta \in \Theta.
	\end{equation}

	Denote by $ \mathbfcal{E}_h := \{F_h(\theta; \cdot)= F(\theta;h^{-1}(\cdot)) \, :\,  \theta \in \Theta \}$   the parameterised family of probability distributions corresponding to random variables $Y(\theta)$, $\theta \in \Theta$. Then, in light of relation \eqref{eq:h-family}, family $\mathbfscr{E}_h$ naturally inherits property \eqref{eq:EXP-property} from the exponential family $\mathbfscr{E}$. More precisely, for every natural number $n$, we define the $h$-convolution operator
	\begin{equation}\label{eq:star-convolution}
		Y(\theta)^{\ostar n} := h\left( X_1(\theta) + \dots + X_n(\theta)\right), \quad \theta \in \Theta,
	\end{equation}
	where the random variables $X_i(\theta) \sim F(\theta;\cdot)$ are independent. We also adapt the conventional notation $F_h(\theta;\cdot)^{\ostar n}$ for the distribution of the random variable $Y(\theta)^{\ostar n}$. Now, it is straightforward to use property \eqref{eq:EXP-property} to find distributional properties of $Y(\theta)^{\ostar n}$. In this vein, note that, for $n \ge 2$,
	\begin{align}\label{eq:hProperty}
		F_h(\theta;x)^{\ostar n} &:= \P \left(   Y(\theta)^{\ostar n} \le x \right) = \P\big(h(X_1(\theta)+ \dots + X_n(\theta)) \leq x\big) = F\big(\theta; h^{-1}(x)\big) ^{\star n} \nonumber\\
		& = \frac{(-\theta)^n}{(n-1)!}  \frac{\partial^{n-1}}{\partial \theta^{n-1}} G(\theta; h^{-1}(x)).
	\end{align}
 It is mentioned that the distributions of the form $Y= h (X_1(\theta) + \dots + X_n (\theta) )$ bears an intimate relation to the inhomogeneous phase-type matrix distributions presented in \cite{ALB2019}.
Hence, it is meaningful, due to the validity of property \eqref{eq:hProperty}, to ask which functions $h$ permit that
	\begin{equation}
		h\left( X_1(\theta) + \dots + X_n(\theta)\right) \stackrel{\dd}{=} h(X_1(\theta)) + \dots + h(X_n(\theta)).
	\end{equation}

	\section{Main results} \label{Sec:Theorems}
	\subsection{$ \mathbfcal{PCFE} (\mu)$:  case of the exponential distribution}\label{sec:ExpCase}

	\begin{thm}\label{Thm:Main}
		Let the following assumptions be fulfilled.
		\begin{itemize}
			\item[$(i)$] The independent random variables $X_1, X_2 \sim \mu$, where $\mu$ is an exponential distribution with parameter $\lambda >0$.
			\item[$(ii)$] The real-valued function $f$ has the following properties:  $f \in C^1[0,+\infty), f(0)=0$, $f(+\infty)=+\infty$, $f^{\prime}(0)>0$ and $f$  is strictly increasing.
			\item[$(iii)$] $f \in \mathbfcal{PCFE}(\mu)$, i.e.,  the following random variables are identically distributed:
			\begin{equation}\label{equ1}
				f\left(X_1+X_2\right) \stackrel{\scriptsize\dd}{=} f\left(X_1\right)+f\left(X_2\right).
			\end{equation}
		\end{itemize}
		Then $f(x)=c x$ for some $c \geqslant 0$.
	\end{thm}

	\begin{rem}\label{rem1}  \begin{itemize}
			\item[$(a)$] Assuming that $f \in C^1[0,+\infty)$, we consider the right-hand  derivative at zero, as usual.
			
			\item[$(b)$] If we assume statements $(i)$ and $(iii)$ from the above theorem as well as that $f$ is strictly increasing and $f \in C^1[0,+\infty)$, then automatically $f(0)=0$ and $  f(+\infty)=+\infty$ otherwise \eqref{equ1} is obviously violated. Indeed, if,  for example, $f(+\infty)=f_\infty\in\mathbb{R}$, then $\P\left(f\left(X_1\right)+f\left(X_2\right)>f_\infty \right)>0$ whilst $\P\left(f\left(X_1+ X_2\right)>f_\infty \right)=0$.
			\item[$(c)$] We can consider symmetric assumption:  $f \in C^1[0,+\infty), f(0)=0$, $f(+\infty)=-\infty$, and $f$ is strictly decreasing. Then we shall get that   $f(x)=c x$ for some $c \leqslant 0$, for which the proof will be the same, but with obvious symmetric changes.
		\end{itemize}
	\end{rem}
	\begin{proof}
		Note that \eqref{equ1} is equivalent to   condition
		$\P(A_x)=\P(B_x)$ for any $x \in \mathbb{R}$, where
		$$
		A_x=\left\{f\left(X_1+X_2\right) \geqslant x\right\}, \quad  B_x=\left\{f\left(X_1\right)+f\left(X_2\right) \geqslant x\right\} \text {. }
		$$
		
		Consider these events separately. Since $f(0)=0,  f(+\infty)=+\infty$ and $f$ is strictly increasing, there exists the inverse function $f^{-1}$ that is also strictly increasing on  $\mathbb{R}^+=[0,  +\infty)$,  $f^{-1}(0)=0$, $f^{-1}(+\infty)=+\infty$.   Then for any $x>0$
		$$
		\begin{aligned}
			A_x =& \; \left\{X_1+X_2 \geqslant f^{-1}(x)\right\} \notag \\
			=& \; \left\{X_2 \geqslant f^{-1}(x)\right\} \cup \left\{0<X_2<f^{-1}(x), X_1 \geqslant f^{-1}(x)-X_2\right\} \text {, }
		\end{aligned}
		$$
		and
		$$
		\begin{aligned}
			B_x =& \; \left\{f\left(X_1\right) \geqslant x-f\left(X_2\right)\right\} \notag \\
			=& \; \left\{X_2 \geqslant f^{-1}(x)\right\} \cup \left\{ X_2<f^{-1}(x), X_1 \geqslant f^{-1}\left(x-f\left(X_2\right)\right)\right\} .
		\end{aligned}
		$$
		Therefore,
		
		$\P\left(A_x\right)=\P\left(B_x\right)$ if and only if
		\begin{equation}\label{Thm1:Eq1}
			\P\left\{0<X_2<f^{-1}(x), X_1 \geqslant f^{-1}(x)-X_2\right\}
			=\P\{0<X_2<f^{-1}(x), X_1 \geqslant f^{-1}(x-f\left(X_2\right))\} .
		\end{equation}
		In terms of exponential densities, Eq.~\eqref{Thm1:Eq1} can be rewritten in the following equivalent form
		\begin{equation}\label{equ3}
			\int_0^{f^{-1}(x)} \lambda e^{-\lambda u} \int_{f^{-1}(x)-u}^{\infty} \lambda e^{-\lambda v} \dd v \dd u=\int_0^{f^{-1}(x)} \lambda e^{-\lambda u} \int_{f^{-1}(x-f(u))}^{\infty} \lambda e^{-\lambda v} \dd v \dd u,
		\end{equation}
		for any $x > 0$. In turn,   it follows from   equation \eqref{equ3} that
		\begin{equation*}
			\int_0^{f^{-1}(x)} \textnormal{exp} \{-\lambda u-\lambda (f^{-1}(x)-u)\} \dd u =\int_0^{f^{-1}(x)} \textnormal{exp} \{-\lambda u-\lambda f^{-1}(x-f(u)) \} \dd u ,
		\end{equation*}
		and hence that
		\begin{equation}
			f^{-1}(x) e^{-\lambda f^{-1}(x) }=\int_0^{f^{-1}(x)} \textnormal{exp} \{-\lambda u-\lambda f^{-1}(x-f(u)) \} \dd u, \label{Thm1:Eq2}
		\end{equation}
		for any $x>0$, because  $$ -\lambda u-\lambda (f^{-1}(x)-u)=-\lambda  f^{-1}(x)$$ and does not depend on $u$.
		
		\medskip
		
		Now, since $f^{-1}(x)$ can be an arbitrary number $y\ge 0$,  we can rewrite Eq.~\eqref{Thm1:Eq2} as
		\begin{equation*}
			e^{-\lambda y} \cdot y=\int_0^y \exp \{-\lambda u-\lambda f^{-1}(f(y)-f(u))\}  \dd u,  %\label{Thm1:Eq3}
		\end{equation*}
		for any $y \geq 0$. Furthermore, by implementing the change of variables $y-u=v$ under the sign of the right-hand side integral, the latter can be rewritten as
		\begin{equation*}
			e^{-\lambda y} \cdot y  = \int_0^y \exp \left\{-\lambda(y-v)-\lambda f^{-1}(f(y)-f(y-v))\right\} \dd v,
		\end{equation*}
		or, equivalently as
		\begin{equation}\label{equ4}
			y = \int_0^y \exp \left\{\lambda v-\lambda f^{-1}(f(y)-f(y-v))\right\} \dd v.
		\end{equation}
		As a consequence of the  equation \eqref{equ4}, we can state that
		\begin{equation}
			\psi(y):=\int_0^y\left(\exp \left\{\lambda v-f^{-1}(f(y)-f(y-v))\right\}-1\right) \dd v=0, \quad y\ge 0, \label{Thm1:psi_quantity}
		\end{equation}
		which will be our function of interest.
		
		\medskip
		
		Since we assume that $f \in C^1[0, +\infty)$ and since $\psi(0)=0$, Eq.~\eqref{Thm1:psi_quantity} is equivalent to the following equation: $\psi^{\prime}(y)=0$. Now, we take explicitly the derivative of the right-hand side of Eq.~\eqref{Thm1:psi_quantity} and take into account that
		$$
		%\begin{aligned}
		\exp \left\{\lambda y-\lambda f^{-1}(f(y)-f(y-y))\right\}-1 =\exp \left\{\lambda y-\lambda f^{-1}(f(y))\right\}-1=0 .
		%\end{aligned}
		$$
		
		Therefore, by recalling for continuously differentiable functions that
		$$\partial_y \Bigl(\int_0^y \varphi(y, v) \dd v \Bigr) = \int_0^y \partial_y \varphi(y, v) \dd v + \varphi(y, y),$$
		and recalling also the formula for the differentiation of the inverse function, we get that
		\begin{equation}
			\psi^{\prime}(y)=\int_0^y \exp \{\theta(y, v)\} \times  \frac{f^{\prime}(y)-f^{\prime}(y-v)}{f^{\prime}(f^{-1}(f(y)-f(y-v)))} \dd v=0 , \label{Thm1:DerivativePsi_Quantity}
		\end{equation}
		where
		$$
		\theta(y, v)=\lambda v-\lambda f^{-1}(f(y)-f(y-v)) .
		$$
		
		Note that $f^{\prime}(x)>0$ for any $x>0$ because the function $f$ is strictly increasing, and $f^{\prime}(x)>0$ by assumption of the theorem. Therefore $$f^{\prime}\left(f^{-1}(f(y)-f(y-v))\right)>0,$$
		for any $v\in[0,y]$, and all components under the sign of integral are continuous implying that the integral is correctly defined.

		\medskip
		
		Our goal now is to prove that the numerator $f^{\prime}(y)-f^{\prime}(y-v)$  is a zero function.  Thus, consider the behavior of $f^{\prime}(y)$, for $y>0$. Of course, there can only be the two following possibilities.
		
		\begin{enumerate}
			\item[$(i)$]  There exists a strictly increasing sequence of points $y_n^{u p},\,n\ge 1$ such that $y_n^{u p} \uparrow \infty$ as $n \rightarrow \infty$, and    $$f^{\prime}\left(y_n^{u p}\right) \ge f^{\prime}(y) \quad \text{for all} \quad y \leq y_n^{u p}, \quad \text{and} \quad f^{\prime}\left(y_{n+1}^{u p}\right) > f^{\prime}\left(y_n^{u p}\right),\, n\ge 1.
			$$
			Since by Eq.~\eqref{Thm1:DerivativePsi_Quantity}
			$$\exp \{\theta(y, v)\}>0 \quad \text{and} \quad f^{\prime}(f^{-1}(f(y)-f(y-v)))>0,
			$$
			we conclude that $\psi^{\prime}\left(y_n^{u p}\right)$ can be zero only if $f^{\prime}(y_n^{u p})=f^{\prime}(y_n^{u p}-v)$ for all $v \in[0, y_n^{u p}]$; i.e. $f^{\prime}(y)=c_n$ for all $y \leq y_n^{u p}$ and some $c_n>0$. Changing $n$ for $n+1$, we conclude that $c_n=c_{n+1}, $ and so $f^{\prime}(y)=c_n=c_1$ for all $y \leq y_{n+1}^{u p}$.
			
			Finally, since $y_n^{up} \uparrow \infty$ as $n \rightarrow \infty$, we get  that  $c_n=c$ for some $c>0$, and $f^{\prime}(y)=c$ for all $y \geq 0$. Therefore, $f(y)=c y$ for $y \geq 0$.
			
		Note that  the case of symmetric downwards conditions (see Remark \ref{rem1}, $(c)$) with $y_n^{down} \uparrow \infty$, where $ f^{\prime}\left(y_n^{down}\right) \leq
			f^{\prime}(y)$ for all $y \leq y_n^{down}$ and $f^{\prime}\left(y_{n+1}^{down}\right) < f^{\prime}\left(y_n^{down}\right)$, can be considered similarly.
			
			\item[$(ii)$] There exists two points $y_0^{up}$ and $y^{down}_0$ such that $f^{\prime}\left(y_0^{up }\right) \geq f^{\prime}(y)$ for
			all $y \geq 0$ and $f^{\prime} ({y}_0^{down} ) \leq f^{\prime}(y)$ for all $y \geq 0$.
			
			Let $y_0^{up}<{y}_0^{down}$ (the case ${y}_0^{up}={y}_0^{{down}}$ is trivial, and the case ${y}_0^{u p}>{y}_0^{down}$ can be considered similarly). Then we can
			"forget" about $y_0^{up}$, consider the whole interval $[0,{y}_0^{down}] $ and  deduce exactly as in   $(i)$, case $(b)$,  that $f^{\prime}(y)=c$ for all $0 \leq y \leq  {y}_0^{down}$.  However, it means immediately that
			$$f^{\prime}({y}_0^{up})=f^{\prime}({y}_0^{down}),$$ and therefore $f^{\prime}(y)=c$ for all $y\ge 0.$
		\end{enumerate}
		
	\end{proof}

	\begin{rem}
		If it is assumed that $f \in C^2[0,\infty)$, then the assumption $f^{\prime}(0)>0$ can be avoided. Indeed, suppose for some $\alpha \in(0,1)$ that $\lim _{x \rightarrow 0} \frac{f^{\prime}(x)}{x^\alpha}=c_\alpha \in (0,\infty)$. Then by l'Hospital's rule
		$$
		\begin{aligned}
			\lim _{x \rightarrow y} \frac{f^{\prime}(y)-f^{\prime}(x)}{f^{\prime}\left(f^{-1}(f(y)-f(x))\right)} =& \; c_\alpha^{-1}\lim _{x \rightarrow y} \frac{f^{\prime}(y)-f^{\prime}(x)}{\bigl[f^{-1}(f(y)-f(x))\bigr]^\alpha}\\
			=& \;  (\alpha c_\alpha)^{-1} \lim _{x \rightarrow y} \frac{-f^{\prime \prime}(x)\bigl[f^{-1}(f(y)-f(x))\bigr]^{1-\alpha}}{-\left[f^{\prime}\left(f^{-1}(f(y)-f(x))\right)\right]^{-1} f^{\prime}(x)}  \quad   \\
			=& \; (\alpha c_\alpha)^{-1}\frac{f^{\prime \prime}(y)}{f^{\prime}(y)} \lim _{x \rightarrow y} \bigl[f^{-1}(f(y)-f(x))\bigr]^{1-\alpha} f^{\prime}(f^{-1}(f(y)-f(x))) \\
			=& \; 0,
		\end{aligned}
		$$
		for any $y>0$,  and where integral \eqref{Thm1:DerivativePsi_Quantity} is correctly defined.
		
	\end{rem}
	
	We also observe that the assumption of independence of $X_1, X_2 \sim \textnormal{Exp}(\lambda)$ for $\lambda>0$ yields for any suitable (e.g.~bounded) $g(x_1,x_2): \R^2 \rightarrow \R$ that
	$$
	\E g(X_1,X_2) = \int_{\R^2} g(x_1,x_2) \lambda^2 \ee^{-\lambda(x_1 + x_2)} \dd x_1 \dd x_2.
	$$
	Using this, we have the following.

	\begin{prop}
		Let $f:\R_+ \rightarrow \R_+ \in C[0,\infty)$ s.t.~$f(x + y) \geq f(x) + f(y)$ for all $x,y \geq 0$. Assume further that $f \in \mathbfcal{PCFE}(\mu)$ where $\mu$ is an exponential distribution with parameter $\lambda>0$. Then $f(x) = cx$ for some $c > 0$.
	\end{prop}
	
	\begin{proof}
		Let $X_1, X_2 \sim \textnormal{Exp}(\lambda)$ be two independent random variables. Consider
		$$
		\E \ee^{-f(X_1+X_2)} = \int_{\R^2} \ee^{-f(x+y)} \lambda^2 \ee^{-\lambda(x + y)} \dd x \, \dd y,
		$$
		and
		$$
		\E \ee^{-f(X_1) -f(X_2)} = \int_{\R^2} \ee^{-f(x) - f(y)} \lambda^2 \ee^{-\lambda(x + y)} \dd x \, \dd y,
		$$
		Then $\E \ee^{-f(X_1+X_2)}  =  \E \ee^{-f(X_1) -f(X_2)}$ by the assumption that $f(X_1 + X_2) \stackrel{\scriptsize \dd}{=} f(X_1) + f(X_2)$, but
		$$
		\int_{\R^2} \ee^{-f(x+y)-\lambda x - \lambda y} \dd x \, \dd y \leq \int_{\R^2} \ee^{-f(x) - f(y)-\lambda x - \lambda y} \dd x \, \dd y,
		$$
		because the $\ee^{-f(x+y)-\lambda x - \lambda y} \leq \ee^{-f(x) - f(y)-\lambda x - \lambda y}$ for any $(x,y) \in \R^2$.  Both statements can only simultaneously be true if $f(x+ y) = f(x) + f(y)$, for all $x, y$. Consequently, by the continuity of $f$, we have that $f(x) = cx$ for some $c >0$.
	\end{proof}
	
	\begin{rem}
		The case where $f:\R_+ \rightarrow (-\infty,0]$ and $f(x + y) \leq f(x) + f(y)$ can be considered similarly.
	\end{rem}

We conclude this section by presenting two results regarding general probability measures, leveraging the assumption of sub(super)additivity within the solution function $f$.

		\begin{prop}%\label{prop:sublinear-generalcase}
			Let $f \in \mathbfscr{PCFE}(\mu)$ be a continuous solution to the probabilistic Cauchy functional equation \eqref{eq:PCE}. Assume further that $f:\R_+ \rightarrow \R_+$ is strictly increasing, $f(0) = 0$, $f(+\infty) = + \infty$ and  that $f(x+y) \le (\ge ) f(x) +f(y)$ for all $x ,y$. Then $f(x) =cx$ for some constant $c$.
		%	\marginpar{\textcolor{blue}{\scriptsize we need to ensure that $\mathbfscr{CFE}$ has the correct domain for this prop. }}
		\end{prop}

		\begin{proof}
			Let $p$ denote the density of $\mu$. In this case, we observe that Eq.~\eqref{equ3} takes the form
			\begin{equation}
				\int_0^{f^{-1}(x)} p(u) \int_{f^{-1}(x)-u}^{\infty} p(v) \dd v \dd u=\int_0^{f^{-1}(x)} p(u) \int_{f^{-1}(x-f(u))}^{\infty} p(v) \dd v \dd u. \notag
			\end{equation}
			Denoting $f^{-1}(x) = y \geq 0$, we get from the above equation that
			\begin{equation}
				\int_0^{y} p(u) \int_{y-u}^{\infty} p(v) \dd v \dd u=\int_0^{y} p(u) \int_{f^{-1}(f(y)-f(u))}^{\infty} p(v) \dd v \dd u, \notag
			\end{equation}
			or equivalently that
			\begin{equation} \label{eq:FundamentalEqforProofofProp}
				\int_0^{y} p(u) \int^{f^{-1}(f(y)-f(u))}_{y-u} p(v) \dd v \dd u = 0.
			\end{equation}
			Assuming that $f(x+y) \geq f(x) + f(y)$ for all $x,y \geq 0$, we observe that $f(y-u) \leq f(y) - f(u)$ for all $0 \leq u \leq y$, whence $y - u \leq f^{-1}\bigl(f(y) - f(u) \bigr)$ for all $0 \leq u \leq y$. Therefore, Eq.~\eqref{eq:FundamentalEqforProofofProp} can hold only for the case when $f(y -u) = f(y) - f(u)$ for all $0 \leq u \leq y$.
		\end{proof}

\begin{prop}\label{prop:sublinear-generalcase}
Let $n \ge 2$ be an integer and $f \in \mathbfscr{PCFE}(\mu;n)$ be a continuous solution to the probabilistic Cauchy functional equation \eqref{eq:PCE} such that $\E_\mu \vert f \vert < \infty$. Assume further that $f(x+y)
\ge (\le) f(x) +f(y)$ for all $x ,y$. Then $f(x) =cx$ for some constant $c$.
\end{prop}

\begin{proof}
Let $p$ denote the density of $\mu$. Then, the validity of equation \eqref{eq:PCE}, the integrability assumption and a direct conditioning argument yield that
\begin{equation}\label{eq:star}
\int_{S(\mu)^n}  \big\{       f(x_1+\dots+x_n) - \left(  f(x_1) +\dots+ f(x_n) \right)        \big  \}	 p(x_1) \dots p(x_n) \dd x_1 \dots \dd x_n=0.
\end{equation}
Now, the assumption of super(sub)-linearity together with \eqref{eq:star} implies that the function
\begin{equation}\label{eq:starstar}
g:S(\mu)^n  \to \R; \, (x_1,\dots,x_n) \mapsto f(x_1+\dots+x_n)- \left( f(x_1) + \dots +f(x_n) \right)
\end{equation}	
is zero almost everywhere for all $(x_1,\dots,x_n) \in S(\mu)^n (=\R^n, \R^n_{+}, \text{ or } \R^n_{-} \text{ by our assumption})$.  But, note that $g$ is a jointly continuous function, and, hence, it is identically zero everywhere. Finally, using the continuity of $f$ along with the simple fact that $f(0)=0$, we immediately conclude that $f(x+y) = f(x) + f(y)$ for all $(x,y)$, and thus the claim follows at once by employing the classic Cauchy functional equation.
\end{proof}

\subsection{$\mathbfcal{PCFE}(\mu)$: a partial result in the general case and its connection with integrated Cauchy functional equations}\label{sec:GeneralCase}
In this section, we provide a partial result in the general case beyond the exponential distribution via tools available in the realm of integrated Cauchy functional equations (ICFEs). In order to briefly motivate the latter,  let $X \sim \mu$ be a random variable distributed as $\mu$ being an absolutely continuous probability measure  on the real line with density $p$. Under the integrability assumption $\E \vert f(X) \vert < \infty$, one can easily see that the probabilistic Cauchy functional equation \eqref{eq:PCE} (equating the first moment of both sides with $n=2)$ yields that

\begin{equation}\label{eq:FirstMomentImplication}
	\int_{S(\mu)} \int_{S(\mu)} \big\{      f(x+y) - \left(   f(x) + f(y) \right)    \big \} p(x) p(y) \dd x \dd y =0.
\end{equation}	
Note that, obviously, any integrable solution $f \in \mathbfcal{PCFE}(\mu)$ is a solution to integral equation \eqref{eq:FirstMomentImplication}, however the inverse implication is far from being true.  For example, consider $\mu = \mathcal{N}(0,1)$, a standard Gaussian distribution. Then, one can easily check that  function $f(x)=x^3$ is a solution to equation \eqref{eq:FirstMomentImplication} but not a solution to the probabilistic Cauchy functional equation \eqref{eq:PCE}.

Now observe that a sufficient requirement for the validity of the integral equation \eqref{eq:FirstMomentImplication} is that, for almost all $x$,
\begin{equation}\label{eq:TowardsFirstMomentSolution1}
	f(x) + \E f(X)  = \int_{S(\mu)}  f(x+y) \mu(\dd y).
\end{equation}	
The latter requirement \eqref{eq:TowardsFirstMomentSolution1} may be viewed as an integrated version of the additive Cauchy functional equation $f(x+y)=f(x)+f(y)$ that appears as a variant of the  so-called \textit{Lau-Rao theorem} (see Section 2 of \cite{Rao1}). The origin of the Lau-Rao theorem can be traced back to the celebrated Choquet-Deny convolution equation $\mu = \mu \ast \sigma$ (see \cite{Choquest-Deny} for details). For the sake of completeness, below, we present a version taken from \cite[Section 4]{Rao1}, slightly adapted to fit our framework.

\begin{thm}[\textbf{The additive Lau-Rao Theorem}]\label{thm:Lau-RaoTheorem}
Let $\mu$ be a non-arithmetic $\sigma$-finite measure on $\R_+$ such that $\mu(\{ 0\}) <1$. Consider the following equation for almost all $x \in \R_+$ (w.r.t Lebesgue measure):
\begin{equation}\label{eq:Lau-RaoEquation}
   \int_{\R_+}  f(x+y) \mu(\dd y) = 	f(x) + \E f(X),
\end{equation}	
where $f:\R_+ \to \R$ is a locally integrable Borel function.  Assume that $f$ is either an increasing or decreasing function almost everywhere. Then, every  solution to the equation \eqref{eq:Lau-RaoEquation} takes the form

\begin{equation}\label{eq:Lau-RauSolutionForm}
f(x) = \begin{cases}
	a + b \left(  1 - e^{-\eta x}\right) & \mbox{ if } \, \eta \neq 0\\
	cx +d  &\mbox{ if } \, \eta =0
\end{cases}	
\end{equation}	
where $a,b,c,d$ are constants and $\eta$ is such that $\int_{\R_+} e^{-\eta x} \mu (\dd x) =1$.  Furthermore, assume that all the assumptions above prevail apart from that instead of $\R_+$, we consider the whole real line $\R$. Then, every solution to the equation \eqref{eq:Lau-RaoEquation} takes the form $\lambda_1 f_1 + \lambda_2 f_2$ where the non-negative constants $\lambda_1$ and $\lambda_2$ are so that $\lambda_1 + \lambda_2 =1$ and, $f_1$ and $f_2$ are of the form \eqref{eq:Lau-RauSolutionForm} with $\eta$ replaced respectively with $\eta_1$ and $\eta_2$ satisfying in equation $\int_{\R} e^{-\eta_i x} \mu (\dd x) =1$, $i=1,2$.
\end{thm}

\begin{rem} \label{rem:LaplaceTransformEtas}
	In our case, we have that $\mu$ is a probability distribution with a support $S(\mu) = \R_+,\R_-$ or $\R$.
	For $S(\mu) = \R_+, \R_-$, we have that $\int_{S(\mu)} \ee^{-\eta x} \mu(\dd x) = 1$ occurs only for $\eta = 0$. However, for $S(\mu) = \R$, we have that $L_\mu (\eta) := \int_\R \ee^{-\eta x} \mu(\dd x) = 1$ for $\eta = 0$ and possibly one other $\eta \neq 0$.
	
	Indeed, let $\mathcal{D}_\eta := \{ \eta \in \R : L_\mu(\eta) < \infty\}$. First note that $\mathcal{D}_\eta$ may be a singleton; for instance, the Cauchy distribution has $\mathcal{D}_\eta = \{0\}$. If, however, some $\eta_u > 0$ is in $\mathcal{D}_\eta$, then since we have for $s \in [0,\eta_u]$ that $\ee^{-(\eta_u - s) x}  \leq \ee^{-\eta_u x} $ on $x \in \R_-$ and $\ee^{-s x} \leq 1 $ on $x \in \R_+$, it must be that $L_\mu (\eta)$ exists at least on the interval $[0,\eta_u]$. A similar reasoning shows that if $\eta_d < 0$ is in $\mathcal{D}_\eta$, then so is the interval $[\eta_d,0]$. Therefore, $\mathcal{D}_\eta$ must itself be some interval (not necessarily including the endpoints).
	
	Furthermore, since the convexity of $\ee^x$ yields that $\ee^{\lambda x_1 + (1-\lambda)x_2} \leq \lambda \ee^{x_1} + (1-\lambda) \ee^{x_2}$ for $x_1,x_2 \in \mathcal{D}_\eta$, $\lambda \in [0,1]$, taking expectations of this inequality yields that $L_\mu (\eta)$ is convex. In turn, since $L_\mu (\eta) = 1$ occurs for $\eta = 0$, it must be that $L_\mu (\eta) = 1$ occurs only for one other $\eta \neq 0$ at most. Indeed, assuming existence of two nonzero points, we get from convexity of $L_\mu$ the existence of the whole interval $  [0,\eta_u)$ where $\E \ee^{-\eta X} = 1$ for all $\eta \in [0,\eta_u)$. Then we can choose $2 \eta \in [0,\eta_u)$ such that $\E \ee^{-2\eta X} = 1$. Therefore, $\text{Var}(\ee^{-\eta X}) = \E \ee^{-2\eta X} - \bigl( \E \ee^{-\eta X} \bigr)^2 = 0$ showing that $X = 0$ a.s.~A simple example of existence of two points where $L_{\mu}(\eta) =1$  is the distribution $\mu \sim \mathcal{N}(\gamma, \sigma^2)$ with mean $\gamma \neq 0$, variance $\sigma^2$, having $L_{\mu}(\eta) = \exp\{ - \gamma \eta + \sigma^2\eta^2/2 \}$.
\end{rem}

%In fact, there is another variant of an integrated version of a Cauchy functional equation which can be obtained by integrating the multiplicative Cauchy functional equation $f(x+y) = f(x)f(y)$ with respect to variable $y$ against a suitable measure $\sigma$. The latter integrated version appears in the so-called \textit{Lau-Rao theorem} (see Section 2 of \cite{Rao1}). To elaborate, the Lau-Rao theorem seeks for (locally integrable) solutions $f:  \R_+ \to \R$ to the equation
%\begin{equation}\label{eq:ICFE1}
%	f(x) = \int_{\R_+} f(x+y) \sigma(\dd y), \quad \text{ for almost all }   x,
%\end{equation}
%where $\sigma$ is a $\sigma$-finite Borel measure on $\R_+$  fulfilling certain non-degeneracy conditions. The root of the Lau-Rao theorem can be traced back to the celebrated Choquet-Deny convolution equation $\mu = \mu \ast \sigma$ (see \cite{Choquest-Deny} for details). It is straightforward to see that by taking the logarithm, any positive solution to the classical multiplicative Cauchy equation provides a solution to the additive Cauchy functional equation. It may be tempting to believe that there is a neat connection between the integrated version of the Cauchy functional equations (as discussed above) as well. However, the following simple example, taken from \cite{Rao1}, illustrates that this is not the case. Let $\mu$ be a probability measure on the real line with zero mean and finite variance. Now, it is simple to observe that function $f(x)=x^2$ satisfies in equation
%\begin{equation*}
 %f(x) + \E_\mu f = 	\int_{\R} f(x+y) \mu (\dd y), \quad x\in \R.
%\end{equation*}

ICFEs have played a significant role in characterisation problems appearing in non-contemporary probability theory and mathematical statistics. For details, the reader can consult textbook \cite{FunctionalEquation-Book}.  The next proposition offers a partial result to our central question by employing this approach.

\begin{thm}\label{prop:GeneralCaseICFE-1}
Fix an integer $n \ge 2$. Let $\mu$ be a non-degenerate probability measure with support $S(\mu)$ having finite second moment. Let $f: S(\mu) (=\R_+, \R_-, \text{or } \R) \to \R \in L^2 (\mu)$ be a Borel measurable function so that with $X \sim \mu$ it holds that for almost all $x$
\begin{equation*}\label{eq:TechnicalCondition}
\hspace{4cm} \E \left(  f(x+X) -f(X)  \right) \ge (\le) f(x).  \hspace{5cm}{ \mathbf{(H)}}
\end{equation*}	
Assume that $f \in \mathbfcal{PCFE}(\mu;n)$ be either an increasing or decreasing function almost everywhere.  Then, $f(x)=cx$ for some constant $c$.
\end{thm}

\begin{proof}[Proof of Theorem \ref{prop:GeneralCaseICFE-1}]
	First, let $X_1, \dots, X_{n-1}$ be $n$ mutually independent random variables having distribution $\mu$. Then, using the tower property and by repetitively using assumption \textbf{(H)}, we obtain that
	\begin{align}
		\E & \bigl(f(x+X_1+\cdots + X_{n-1})\bigr) \notag
		= \E\bigl( \E\bigl( f(x+X_1+\cdots + X_{n-1}) \big| X_1, \dots, X_{n-2} \bigr)\bigr) \notag \\
		&\geq (\leq) \E\Big( \E\bigl( f(x+X_1+\cdots + X_{n-2}) \big| X_1, \dots, X_{n-3} \bigr) + f(X_{n-1})\Big) \notag \\
		&\;\;\;\;\; \vdots \notag \\
		&\geq (\leq) \E\Big( \E\bigl( f(x+X_1 \bigr) + f(X_2) + \cdots + f(X_{n-1})\Big) \notag \\
		&\geq (\leq) \E\bigl( f(x) + f(X_1) + \cdots + f(X_{n-1})\bigr) \notag.
	\end{align}
	The latter is equivalently expressed as
	\begin{equation}\label{eq:Assumption H(n)}
		\int_{S(\mu)^{n-1}} \Big(    f(x+x_1+\dots+x_{n-1})  - \left( f(x)+f(x_1)+\dots +f(x_{n-1})  \right)    \Big) \mu(\dd x_1) \dots \mu (\dd x_{n-1}) \ge (\le ) 0.
	\end{equation}
	Now, assume that $f \in \mathbfcal{PCFE}(\mu;n)$ such that $f \in L^2(\mu)$. By equating the expectations of both sides of equation \eqref{eq:PCE}, we obtain
	\begin{equation}
		\int_{S(\mu)^{n}} \Big(    f(x+x_1+\dots+x_{n-1})  - \left( f(x)+f(x_1)+\dots +f(x_{n-1})  \right)    \Big) \mu(\dd x_1) \dots \mu (\dd x_{n-1})  \mu (\dd x) = 0, \notag
	\end{equation}
	and so the above statement combined with Eq.~\eqref{eq:Assumption H(n)} yields for almost all $x$ that
	\begin{equation}\label{eq:1}
		f(x) + (n-1) \E f(X) = \int_{S(\mu)^{n-1}} f(x+x_1+\dots+x_{n-1}) \mu(\dd x_1) \dots \mu (\dd x_{n-1}).
	\end{equation}
	Now, for $n \geq 3$, observe from Eq.~\eqref{eq:1} and assumption \textbf{(H)} that
	\begin{align*}
		f(x) + (n-2) \E f(X)
		&= \E\bigl( f(x+X_1+\cdots +X_{n-2} + X_{n-1}) - f(X_{n-1}) \bigr) \notag \\
		&= \E\Big( \E \bigl( f(x+X_1+\cdots +X_{n-2} + X_{n-1}) - f(X_{n-1}) \big| X_1,\dots X_{n-2} \bigr) \Big) \notag \\
		& \geq (\leq) \E\bigl( f(x+X_1+\cdots +X_{n-2}) \bigr),
	\end{align*}
	and so repeating the above procedure yields that $f(x) + \E f(X) \geq (\leq) \E\bigl( f(x+X) \bigr)$ where $X \sim \mu$.  Hence,	combining the latter inequalities with assumption \textbf{(H)}, we obtain that for almost all $x$,
	\begin{equation}
		f(x) + \E f(X) = \int_{S(\mu)} f(x+y) \mu(\dd y). \label{eq:2}
	\end{equation}
	Now, we consider two cases separately.\\
	 \textit{Case $S(\mu) =\R_+$}: First we note that equation $L_\mu (\eta ): = \int_{\R_+} e^{- \eta x} \mu (\dd x)=1$ has only trivial solution $\eta =0$.
	 %\grr{ To see this, let $F$ denote the distribution function corresponding to $\mu$. Then, using integration by parts formula, we obtain, for each $\eta \neq 0$ that
%	 \begin{equation}
%	 	\frac{1 -L_{\mu}(\eta)}{\eta} =  \int_{\R_+} e^{-\eta x} \bigl(1 - F (x) \bigr) \dd x.
%	 \end{equation}
%Hence, the existence of a non-trivial solution to the equation $L_{\mu}(\eta)=1$ leads to a contradiction.}
Now, since it is assumed that $f$ is monotone, a direct application of the Lau-Rao Theorem  \ref{thm:Lau-RaoTheorem} yields that all the solutions to equation \eqref{eq:2} are given by $f(x)= cx +d$. Referring to the probabilistic Cauchy functional equation, we deduce that $d=0$. The case $S(\mu)=\R_-$ follows similarly. \\

\textit{Case $S(\mu) =\R$}: In addition to $\eta = 0$, there may possibly be an additional non-trivial $(\eta \neq 0)$ solution to the equation $L_\mu (\eta)=1$, see Remark \ref{rem:LaplaceTransformEtas}.
%	Suppose that $L_\mu (\eta)$ exists on some domain $\mathcal{D}_\eta$ (we may have that $\mathcal{D}_\eta$ is a singleton; consider for instance the Cauchy distribution which has $\mathcal{D}_\eta = \{0\}$). In addition, if $\eta >0$ ($\eta < 0$) is a non-trivial solution, we observe for $s \in [0,\eta]$ that $\ee^{-(\eta - s) x}  \leq \ee^{-\eta x} $ on $x \in \R_-$ and $\ee^{-s x} \leq 1 $ on $x \in \R_+$, and thus $L_\mu (\eta)$ must exist at least on the interval $[0,\eta]$. Therefore we have that $[0,\eta] \subseteq \mathcal{D}_\eta$. Furthermore, $L_\mu (\eta)$ is convex on $\mathcal{D}_\eta$ which holds by using the convexity of $\ee^x$ to observe that  for $x_1,x_2 \in \mathcal{D}_\eta$ and then taking expectations of this inequality.
%	
%	We must have that $\mathcal{D}_\eta$ is some interval (not necessarily open) since if $L_\mu(\eta)$ exists for $\eta >0$, we observe for $s \in [0,\eta]$ that $\ee^{-(\eta - s) x}  \leq \ee^{-\eta x} $ on $x \in \R_-$ and $\ee^{-s x} \leq 1 $ on $x \in \R_+$, and thus $L_\mu (\eta)$ must exist at least on the interval $[0,\eta]$

	Relying on the Lau-Rao Theorem \ref{thm:Lau-RaoTheorem} since $f$ is monotone by assumption, every solution $f$ to the equation \eqref{eq:2} can be written as a convex combination $f = \lambda_1 f_1 + \lambda_2 f_2$ of two solutions of the form given by \eqref{eq:Lau-RauSolutionForm}. The rest of the proof is dedicated to showing that the only valid solutions to the probabilistic Cauchy equation \eqref{eq:PCE} are given by linear functions. 	Hence, assume that  $\eta_1 \neq 0$ and $\eta_2 = 0$ are the solutions of $L_{\mu} (\eta)=1$, and  consider  (for $\lambda_1 \in [0,1]$ and $\lambda_2 = 1- \lambda_1$)
		\begin{equation}\label{eq:SolutionForm1}
			f(x) = \lambda_1b(1-\ee^{-\eta_1 x}) + \lambda_2 c x,
	\end{equation}
	for some constants $b$ and $c$ that are chosen to ensure the monotonicity of $f$;  for instance, if $\eta <0$, then $b=1$ and $c=-1$ are permissible constants. (We exclude the the constants $a,d$ in \eqref{eq:Lau-RauSolutionForm} since they must equal $0$ for $f$ to satisfy in the probabilitistic Cauchy functional equation.) 
	
	To avoid the trivial cases, let $\lambda_1 \neq 0,1$. Let $\lambda := \lambda_2 c/( \lambda_1 b)$. Since $f \in \mathbfcal{PCFE}(\mu;n)$ and $f \in L^2(\mu)$ by our assumption,  the second moments of $f(\sum_{i=1}^n X_i)$ and $\sum_{i=1}^n f(X_i)$ exist and must be equal where $f$ is given by Eq.~\eqref{eq:SolutionForm1};  i.e.~we have that
	\begin{align}
		\E\Bigl( 1 - \exp(-\eta_1 \sum_{i = 1}^n X_i) + \lambda \sum_{i = 1}^n X_i) \Bigr)^2
		&= \E\Bigl( \sum_{i = 1}^n \bigl[1 - \exp(-\eta_1 X_i) \bigr] + \lambda \sum_{i = 1}^n   X_i \Bigr)^2.  \label{eq:MainSecondMomentIdentity(ii)}
	\end{align}
		Expanding the r.h.s.~of the above equation, we get
	\begin{align}
		&\E\Bigl( \sum_{i = 1}^n \bigl[1 - \exp(-\eta_1 X_i) \bigr] + \lambda \sum_{i = 1}^n   X_i \Bigr)^2  \notag \\
		&= \E\Bigl( \sum^{n}_{i=1} \bigl[1 - \exp(-\eta_1 X_i) \bigr] \Bigr)^2 + \lambda^2 \E\Bigl(  \sum_{i = 1}^n X_i \Bigr)^2 + 2\lambda \E \Bigl(\sum^{n}_{i=1}\bigl[ 1 - \exp(-\eta_1  X_i) \bigr] \sum_{i = 1}^n X_i  \Bigr), \notag \\
		&=   \E\Bigl( \sum^{n}_{i=1} \bigl[1 - \exp(-\eta_1 X_i) \bigr] \Bigr)^2 + \lambda^2 \E\Bigl(  \sum_{i = 1}^n X_i \Bigr)^2 + 2 \lambda \E \Bigl( \sum_{i = 1}^n\bigl[ 1 - \exp(-\eta_1 X_i) \bigr] X_i  \Bigr). \label{eq:RHSQuantity(ii)}
	\end{align}
	and the l.h.s.~yields
	\begin{align}
		&\E\Bigl(  1 - \exp(-\eta_1 \sum_{i = 1}^n X_i) + \lambda \sum_{i = 1}^n X_i) \Bigr)^2 \notag \\
		&= \E\Bigl( 1 - \exp(-\eta_1 \sum_{i = 1}^n X_i) \Bigr)^2 + \lambda^2 \E\Bigl(  \sum_{i = 1}^n X_i \Bigr)^2 + 2\lambda \E \Bigl(\bigl[ 1 - \exp(-\eta_1 \sum_{i = 1}^n X_i) \bigr] \sum_{i = 1}^n X_i  \Bigr), \label{eq:LHSQuantity(ii)}
	\end{align}
	However, by using that $L_\mu (\eta_1)=1$ and the independence assumption of $X_1, ...,X_n$, the last term of Eq.~\eqref{eq:LHSQuantity(ii)} becomes
	\begin{align}
		\E \Bigl(\bigl[ 1 - \exp(-\eta_1 \sum_{i = 1}^n X_i) \bigr] \sum_{i = 1}^n X_i  \Bigr) &= \sum_{i = 1}^n  \E \Bigl(  X_i \bigl[ 1 - \exp(-\eta_1 X_i) \bigr] \Bigr) \notag \\
		& \;\;\;\;\; + \sum_{i = 1}^n  \E \Bigl(  X_i \bigl[ 1 - \exp(-\eta_1 \sum_{j \neq i} X_j) \bigr] \Bigr) \notag \\
		&=   \E \Bigl( \sum_{i = 1}^n X_i \bigl[ 1 - \exp(-\eta_1 X_i) \bigr] \Bigr), \notag
	\end{align}
	which is the same as that in Eq.~\eqref{eq:RHSQuantity(ii)}.
	Therefore, by using the above identity and equating Eq.~\eqref{eq:LHSQuantity(ii)} and \eqref{eq:RHSQuantity(ii)}, we have that Eq.~\eqref{eq:MainSecondMomentIdentity(ii)} reduces to
	\begin{equation}
		 \E\Bigl(  1 - \exp(-\eta_1 \sum^{n}_{i=1} X_i) \Bigr)^2 =  \E\Bigl( \sum^{n}_{i=1} \bigl[1 - \exp(-\eta_1 X_i) \bigr] \Bigr)^2 .
	\end{equation}
	Since $L_\mu (\eta_1)=1$, we observe for the l.h.s.~of the above equation that
	\begin{align}
		\E\Bigl( 1 - \exp(-\eta_1 \sum_{i = 1}^n X_i) \Bigr)^2 &= 	\E\Bigl(  \exp(-2 \eta_1 X_1) \Bigr)^n - 1, \label{eq:LHSQuantityFinal}
	\end{align}
	and, by using independence of the $X_i$ to notice that the expectations of the cross-terms are zero, that the r.h.s.~yields
	\begin{align}
		\E\Bigl( \sum^{n}_{i=1} \bigl[1 - \exp(-\eta_1 X_i) \bigr] \Bigr)^2 &= \sum^{n}_{i=1} \E\Bigl(  \bigl[1 - \exp(-\eta_1 X_i) \bigr] \Bigr)^2 \notag \\
		&= n \Bigl( \E\Bigl( \exp(-2\eta_1 X_1) \Bigr) - 1 \Bigr). \label{eq:RHSQuantityFinal}
	\end{align}
	Now, we shall denote $x:= \E  \exp(-2 \eta_1 X_1) $, and thus we have by equating Eqs.~\eqref{eq:LHSQuantityFinal} and \eqref{eq:RHSQuantityFinal} that
	\begin{equation}
		(x^n - 1)  = n(x -1) , \notag
	\end{equation}
	which is equivalent to
	\begin{equation}
		(x-1)(\sum^{n-1}_{k=0}x^k - n) = 0. \label{eq:Contradiction!!}
	\end{equation}
	If $x < 1$, we have that $\sum^{n-1}_{k=0}x^k < n$, and if $x > 1$, then $\sum^{n-1}_{k=0}x^k > n$, and so \newline $(x-1)(\sum^{n-1}_{k=0}x^k - n) > 0$ if $x \neq 1$. Thus Eq.~\eqref{eq:Contradiction!!} yields that $x=1$. On the other hand, note that $x \neq 1$ yields a contradiction. Therefore, functions $f$ of the form \eqref{eq:SolutionForm1} cannot be a solution to the probabilistic Cauchy equation \eqref{eq:PCE}.

\end{proof}

\begin{rem}\label{rem:ClarifyingSymmetricAssumption}

In the proof of Theorem \ref{prop:GeneralCaseICFE-1}, when the support $S(\mu) = \R$ is the entire real line, following the Lau-Rao Theorem \ref{thm:Lau-RaoTheorem},  we had to treat solutions of the non-linear forms given by \eqref{eq:SolutionForm1} involving a non-trivial $(\eta \neq 0)$ constant $\eta$ so that $L_\mu (\eta) =1$.  However, all those non-linear solutions would be immediately eliminated as soon as measure $\mu$ is symmetric about the origin, i.e., $\mu (A) =\mu (-A)$ or equivalently $F(-x) = 1- F(x)$ where $F$ stands for the corresponding distribution function. In other words, the only solution to the equation $L_\mu (\eta)=1$ is given by the trivial solution $\eta=0$.  To see this,  Theorem 3c from pp.~239 of \cite{Widder1941}  guarantees that $L_{\mu}(\eta) =  \eta \int_{\R} e^{-\eta x} F (x) \dd x$ for $\eta \neq 0$, and provided that the Laplace transform $L_{\mu}(\eta)$ exists.	Hence, using the symmetry assumption, we have for $\eta > 0$ that
	\begin{align}
		\frac{L_{\mu}(\eta)}{\eta} &= \int_0^\infty e^{-\eta x} \bigl( F (x) - 1 + 1) \dd x + \int_{-\infty}^0 e^{-\eta x} F (x) \dd x \notag \\
		&=  -\int_0^\infty e^{-\eta x} \bigl(1 - F (x) \bigr) \dd x + \frac{1}{\eta} + \int^\infty_0 \ee^{\eta x} F(-x) \dd x
		\notag \\
		&= -\int_0^\infty e^{-\eta x} \bigl(1 - F (x) \bigr) \dd x + \frac{1}{\eta} + \int^\infty_0 \ee^{\eta x} \bigl( 1- F(x) \bigr) \dd x, \notag
	\end{align}
	which consequently yields that
	\begin{equation}
		\frac{L_{\mu}(\eta)-1}{\eta} = \int^\infty_0 \bigl(\ee^{\eta x} - \ee^{-\eta x} \bigr)\bigl( 1- F(x) \bigr) \dd x, \notag
	\end{equation}
	a contradiction  if $L_{\mu}(\eta) = 1$ for $\eta > 0$ (the case of $\eta < 0$ can be shown to yield a similar contradiction). A simple example fulfilling this setup is given by the centred Gaussian distribution $\mu = \mathcal{N}(0,\sigma^2)$ with $L_\mu(\eta) = e^{-\sigma^2 \eta^2 /2}$.
\end{rem}

\begin{rem}\label{rem:AssumptionH}
	For a given probability measure $\mu$, assumption $\textbf{(H)}$ can be viewed as a form of super(sub) linearity on average with respect to the measure $\mu$, which, in general, is much weaker than the classic pointwise super(sub) linearity. In other words, for a given function $f$ assumption $\textbf{(H)}$  is equivalent to
	
	\begin{equation}\label{eq:SubSupmu}
		\int_{S(\mu)} f(x+y) \mu (\dd y)  \ge (\le) \int_{S(\mu)} \left(f(x)+f(y) \right) \mu (\dd y), \quad \text{ for almost all $x$}.
	\end{equation}

\end{rem}

%\begin{cor}
%Fix an integer $n\ge 2$. Let $X \sim \mu$ where $ \mu$ may be an exponential, Erlang, Mittag-Leffler distribution (supported on $\R_+$) or standard Gaussian  (supported on $\R$) distribution. Let $f:\R_+  \, (\R) \to \R \in L^1 (\mu)$ be a Borel measurable function satisfying in assumption $\mathbf{(H)}$. Assume that $f \in \mathbfcal{PCFE}(\mu;n)$ may be either an increasing or decreasing function almost everywhere.  Then, $f(x)=cx$ for some constant $c$.
%	\end{cor}

\begin{rem}\label{rem:ClarifyingDifferentConditions}
This remark aims to clarify the relationships between all the assumptions imposed on the solution function $f$ in the aforementioned results.

\medskip
	
(a) It is clear that any function satisfying in $f(x+y) \ge f(x)+f(y)$ satisfies also assumption  $\textbf{(H)}$ no matter the choice of the underlying probability measure $\mu$.  However, the converse is not necessarily true. Indeed, Lemma \ref{lem:CounterExamples} from the Appendix demonstrates a piecewise linear function $f$ satisfying  assumption $\textbf{(H)}$ w.r.t.~the exponential distribution but for which pointwise super-linearity fails for a set which is a non-trivial interval.

	\medskip
	
	(b) Let $\mu$ be an exponential distribution. The requirement that $f:[0,\infty) \mapsto [0,\infty)$ is increasing (decreasing) does not automatically imply the validity of assumption $\mathbf{(H)}$,  and hence both assumptions are required for Proposition \ref{prop:GeneralCaseICFE-1}. For example, consider the function $f(x) =e^{x/2}$ which is clearly (strictly) increasing but violates $\mathbf{(H)}$. Furthermore, Lemma \ref{lem:CounterExamples} shows that the converse does not hold either since a non-increasing function satisfying in assumption $\mathbf{(H)}$ exists by choosing $a=b=c=d=1$ and $r \geq 2e- 1$.
	
	\medskip
	
	(c) Note that for a given function $f:[0,\infty) \to [0,\infty)$ the condition $f(x+y) \geq f(x) + f(y), \; \forall x,y \in \R_+$ implies that $f$ is necessarily increasing but not strictly increasing (consider $f(x)=0$, for $x \in [0,1]$ and $f(x) = x-1$ for $x \ge 1$). Furthermore, the strictly increasing assumption does not imply super-linearity, i.e.,  $f(x+y) \geq f(x) + f(y), \; \forall x,y \in \R_+$. Indeed, consider $f(x) = 2x-\sin x$. Then, finding $x$ and $y$ s.t. $f(x+y) \leq f(x) + f(y)$ is the same as finding $x$ and $y$ s.t. $\sin(x+y) \geq \sin x + \sin y$, and so choosing $x,y \in (\pi, \tfrac{3}{2}\pi)$ is sufficient.
\end{rem}

\bibliographystyle{acm}
\bibliography{Refs_9_Noah}

\section{Appendix}

	\begin{lem}\label{lem:CounterExamples}
	Let $X$ be an exponentially distributed random variable with parameter $\lambda=1$. Then, there exist constants $a,b,c,d,r>0$ and a set $B \subseteq \R^2_+$ (having positive Lebesgue measure) such that the function
		\begin{equation}
			f(x) = \begin{cases}
				ax, & 0 \leq x \leq b \\
				-dx + (a+d)b, & b \leq x \leq b+c \\
				rx - (r+d)(b+c) + (a+d)b, & x \geq b+c \notag
			\end{cases}
		\end{equation}
satisfies
		\begin{equation}
			\E \bigl( f(x+X) - f(X) \bigr) \ge f(x), \quad \forall x \in \R_+, \notag
		\end{equation}
		but $f(x+y) \leq f(x) + f(y)$ for $(x,y) \in B$.
	\end{lem}
	
	\begin{proof}
		Observe that there exists $N \in \N$ s.t. $b/N < c$.  This implies that any $x \in (0,\tfrac{b}{2N})$ and $y \in (b, b + \tfrac{b}{2N})$ satisfies $f(x+y) \leq f(x) + f(y)$ (take $B = (0,\tfrac{b}{2N}) \times (b, b + \tfrac{b}{2N})$).
		%Indeed, choosing $\varepsilon > 2$, $x = \tfrac{b}{\varepsilon N}$ and $y = b + \tfrac{b}{\varepsilon N}$, we have that
		%\begin{align}
		%	f(b + \tfrac{2b}{\varepsilon N}) = ab - d(\tfrac{2b}{\varepsilon N}) < ab - d(\tfrac{b}{\varepsilon N}) + a(\tfrac{b}{\varepsilon N})  =  f(\tfrac{b}{\varepsilon N}) + f(b + \tfrac{b}{\varepsilon N}).\notag
		%\end{align}
		Now, let $\boldsymbol{\theta} := (a,b,c,d,r)$ and $H(x, \boldsymbol{\theta}) := \E \bigl( f(x+X) - f(X) \bigr) - f(x)$. Since $f(0) = 0$, we necessarily have that $H(0,\boldsymbol{\theta}) = 0$; and thus showing that $H^\prime(x,\boldsymbol{\theta}) \geq 0$ holds for some $a,b,c,d,r >0$ is sufficient. Therefore, we have
		\begin{equation}
			H^\prime(x,\boldsymbol{\theta}) = \int^\infty_0  f^\prime(x+y) e^{-y} \dd y - f^\prime (x), \notag
		\end{equation}
		where
		\begin{equation}
			f^\prime(x) = \begin{cases}
				a, & 0 < x < b \\
				-d, & b < x < b+c \\
				r & x > b+c. \notag
			\end{cases}
		\end{equation}
It is now required to find the range of the parameters for which $H^\prime(x,\boldsymbol{\theta}) \geq 0$ is satisfied. First, for $x \in [0,b]$,
		\begin{align}
			H^\prime(x,\boldsymbol{\theta}) &= \int^{b-x}_0  a e^{-y} \dd y - \int_{b-x}^{b-x+c}  d e^{-y} \dd y + \int_{b-x+c}^\infty  r e^{-y} \dd y - a \notag \\
			&= (r+d) e^{-(b-x+c)} - (a+d) e^{-(b-x)}, \notag
		\end{align}
		and hence $H^\prime(x,\boldsymbol{\theta}) \geq 0$ yields that
		\begin{equation}
			c \leq \log (\tfrac{r+d}{a+d}). \label{eq:Lemma-condition-on-c}
		\end{equation}
		Performing similar calculations for $x \in [b,b+c]$ and $x \in [b+c,\infty)$ yields no further information about the parameters. Then, since there always exists some $N \in \N$ such that $b/N < c$ and $a,d,r >0$ s.t.~$c \leq \log (\tfrac{r+d}{a+d})$, these parameters satisfy the conditions of the lemma.
	\end{proof}
	
%	\begin{lem}\label{lem:StrictlyIncreasing}
%		Let $f:[0,\infty) \mapsto [0,\infty)$ be a function not identically zero satisfying $f(x+y) \geq f(x) + f(y), \; \forall x,y \in \R_+$. Then $f$ is strictly increasing.
%	\end{lem}
	
%	\begin{proof}
%		Observe that choosing $x=y=0$ yields that $f(0)=0$.	Now, we proceed by contradiction to show that that $f(x) > 0$ for $x >0$. Thus, suppose that $x^* > 0$ is the smallest value such that $f(x^*) = 0$. Now, for $y < x^*$, the assumption gives that
	%	$$
%		0 = f(x^*) \geq f(x^*-y) + f(y),
%		$$
	%	but we have that $f(x) > 0, \; \forall x \in (0,x^*)$, a contradiction.  Lastly, using that $f(x) > 0$ for $x >0$, we see that the assumption yields for $x > y$ that $	f(x) \geq f(x-y) +f(y) > f(y)$.
	%\end{proof}

\end{document}